\documentclass[12pt]{amsart}
\usepackage{graphicx}

\usepackage[english]{babel}
\usepackage[margin = 1in]{geometry}
\usepackage{amsmath,amssymb,amsthm, epsf, graphics,enumerate,fancyhdr,bbm, inputenc, pgfplots, csquotes, environ, enumitem, xcolor, mathtools, amsfonts, color, latexsym, mathrsfs, pgffor, verbatim}
\usepackage{enumitem}

\usepackage{environ}
\usepackage{etoolbox}
\usepackage{graphicx}

\newlength{\myl}
\let\origequation=\equation
\let\origendequation=\endequation

\RenewEnviron{equation}{
  \settowidth{\myl}{$\BODY$}                 
  \origequation
  \ifdimcomp{\the\linewidth}{>}{\the\myl}
  {\ensuremath{\BODY}}                             
  {\resizebox{\linewidth}{!}{\ensuremath{\BODY}}}  
  \origendequation
}

\newcommand{\N}{\mathbb{N}}     
\newcommand{\R}{\mathbb{R}}
\newcommand{\Z}{\mathbb{Z}}
\newcommand{\Q}{\mathbb{Q}}
\newcommand{\C}{\mathbb{C}}
\newcommand{\norm}[1]{\left\lVert#1\right\rVert}

\numberwithin{equation}{section}

\theoremstyle{plain}
\newtheorem{theorem}{Theorem}
\newtheorem{lemma}{Lemma}

\newtheorem{proposition}{Proposition}

\theoremstyle{definition}
\newtheorem{definition}{Definition}

\newtheorem{question}{Question}

\theoremstyle{remark}

\title{Fourier Dimension and Translation Invariant Linear Equations}
\author{Angel D. Cruz}
\date{\today}

\pgfplotsset{compat=1.18}

\begin{document}

\begin{abstract}
We consider a translation invariant linear equation in four variables with integer coefficients of the form: $ax_1 +bx_2= cy_1+dy_2$. 
The main result of the paper states that any set on the real line with Fourier dimension greater than $1/2$ must contain a nontrivial solution of such an equation. 
\end{abstract}

\maketitle

\section{Introduction}

We investigate the following type of equations:
\begin{equation}\label{4variable equation}
ax_1 + bx_2 = cy_1+dy_2 \text{ where } \{a,b,c,d\}\subset \N \text{ obeying } a+b=c+d  
\end{equation}
and $\{x_1,x_2,y_1,y_2\}\subset \R$. 
Such equations are examples of \textbf{translation invariant linear equations} with integer coefficients.
Translation invariant because any $\{x_1,x_2,y_1,y_2\}$ that satisfies \eqref{4variable equation} and any $t\in \R$ the collection $\{x_1+t,x_2+t,y_1+t,y_2+t\}$ remains a solution to \eqref{4variable equation} by the condition $a+b=c+d$. 
These equations are widely studied in many areas of mathematics including harmonic analysis, combinatorics and number theory; see \cite{ruzsa1}, \cite{MR4277805}, \cite{MR4307997}, \cite{MR2519934}, \cite{MR2823971} and \cite{MR2153903}.
When $a=b=c=d$, \eqref{4variable equation} reduces to $x_1+x_2=y_1+y_2$, an equation central to the study of Sidon sets.
Another special case occurs when $y_1=y_2=y$ but $x_1\not=x_2$, in this case \eqref{4variable equation} reduces to $x_1+x_2=2y$ and $\{x_1,x_2,y\}$ lie in an arithmetic progression. 
This type of equation is also a translation invariant linear equation and is studied in \cite{LiangandPramanik} with respect to the \textit{size} of a set that must contain a \textit{nontrivial solution} to \eqref{4variable equation}.

In our current work, we consider the following question
\begin{multline}\label{question main}
\text{How ``large" must a set $E\subset [0,1]$ be to ensure it contains} \\
\text{a ``nontrivial" solution to an equation of the form \eqref{4variable equation}?}
\end{multline}
The words ``large" and ``nontrivial" are open to interpretation and we start with the latter.
For a given $\{a,b,c,d\}$ obeying $a+b=c+d$ and $E\subset [0,1]$ we define the set of solutions $\mathcal{S}$ of \eqref{4variable equation} in $E$ as follows
\[
\mathcal{S}:=\mathcal{S}(E,\{a,b,c,d\}) = \{(x_1,x_2,y_1,y_2) \in E^4 :ax_1 + bx_2 = cy_1+dy_2 \}.
\]
This set is always nonempty because $x_1=x_2=y_1=y_2$ always satisfies \eqref{4variable equation}.
Such a solution will be called a trivial solution. Loosely speaking, other types of solutions are \textbf{nontrivial}. For example, solutions of \eqref{4variable equation} where $x_1, x_2, y_1, y_2$ are distinct are always nontrivial. See subsection \ref{subsection: nontrivial solutions} for a precise definition of nontrivial solutions. 

We measure the \textit{size} of a set using the Fourier dimension. 

\begin{definition}[Fourier Dimension]\label{Fourier dimension definition}
Let $E\subset \R$ be a Borel set.
The \textit{Fourier dimension} $E$, denoted \text{dim}$_\mathbb{F}(E)$, is defined as 
\[
\text{dim}_\mathbb{F}(E) = \sup \left\{ 0\le \beta\le 1 :
\begin{aligned}
&\exists \text{ a probability measure $\mu$ supported on $E$ and } \\ &\text{$C \in (0, \infty)$ such that }
|\hat{\mu}(\xi)| \le  C(1+|\xi |)^{-\beta/2}  \  \forall \xi \in \R 
\end{aligned}
\right\}. 
\]
\end{definition}

Reformulating our question \eqref{question main}, our objective is to find the best constant $\beta$ such that dim$_\mathbf{F}(E)>\beta$ implies $E$ contains a nontrivial solution to \eqref{4variable equation} for some $\{a,b,c,d\}$.
Our main result in line with this inquiry is displayed below.

\begin{theorem}\label{theorem: containment result}\footnote{A more general version of Theorem \ref{theorem: containment result} holds for multivariate translation-invariant equations. This is currently work in progress \cite{phdthesis}.}
Let $E\subset [0,1]$ be a Borel set with $\text{dim}_\mathbb{F}(E)>\frac{1}{2}$.
Then there exists an equation of the form \eqref{4variable equation} for which $E$ contains a nontrivial solution.
\end{theorem}

In Theorem \ref{theorem: containment result}, $\dim_{\mathbb F}$ cannot be replaced by Hausdorff dimension. This is originally due to a body of work, authored by Keleti \cite{08keleti}, Maga \cite{maga} and Math\'e \cite{MATHE2017691}. Let us expand on these results. Keleti calls a collection of points $\{x_1,x_2,y_1,y_2\}\subset \R$ a rectangle if they satisfy $x_1-y_1=x_2-y_2$(observe this is a translation invariant linear equation) and shows that there exists a full Hausdorff dimensional set in the real line that avoids (nontrivial) rectangles where $x_1 \ne y_1, x_2 \ne y_2$. 
These results were expanded by Maga in \cite{maga} to full dimensional sets in $\R^d$ that are able to avoid the vertices of any parallelogram; Maga again defines parallelograms to be a collection of points $\{ x_1,x_2,x_3,x_4\}\in \R^d$ satisfying $x_2-x_1 = x_4-x_3$.
Such results were considerably generalized  by Math\'e in \cite{MATHE2017691}, which provides a construction of a compact set not containing a solution to countably many given polynomial equations. 
In particular, a special case of Math\'e's result is the existence of a rationally independent set in $\R$ of full Hausdorff dimension, similar to the result of Keleti. 
That is, there exists a set $E\subset \R$ such that for any collection of points $\{x_1,\dots,x_n\}\subset E^n $ satisfying $\sum a_ix_i=0$ with $a_i\in \Z$ for all $i$, it must be that $a_1=\cdots=a_n=0$. 
As is discussed in \cite{LiangandPramanik}, a rationally independent set cannot have positive Fourier dimension; Theorem \ref{theorem: containment result} aligns with this result.

Our next result partially complements theorem \ref{theorem: containment result}.

\begin{theorem}\label{E For avoidance}
Fix $\{a,b,c,d\}\subset\N$ obeying \eqref{4variable equation}.
There exists a set $E\subset [0,1]$ of Fourier dimension $1/2$ that does not contain any nontrivial solutions of \eqref{4variable equation}. 
\end{theorem}

It is unclear at the moment if the dimensional threshold 1/2 is sharp.
Our methods rely on a transference principle (see section \ref{subsection: transference}) of an avoiding set from the discrete to the continuum, and as such, the Fourier dimension is directly tied to known cardinality bounds for such sets. 
The following question is of interest in this regard.

\begin{question}
A set $E\subset [0,1]$ is Sidon if there does not exist a nontrivial collection $\{x,y,u,v\}\subset E$ obeying $x+y=u+v$.
What is the maximal Fourier dimension of a Sidon set in the real line?
\end{question}

Theorem \ref{E For avoidance} shows that Sidon sets can have Fourier dimension at least 1/2 but does not rule out the possibility of a Sidon set having Fourier dimension larger than $1/2$.

\subsection{Nontrivial Solutions and Genus}\label{subsection: nontrivial solutions}

In this subsection we give a precise definition of \textbf{nontrivial solutions} to an equation like \eqref{4variable equation}. 
This relies on the notion of \textbf{genus} given in \cite{ruzsa1}, which we include here. 

Let us write \eqref{4variable equation} in the following way 
\[
ax_1+bx_2+(-c)y_1 +(-d)y_2 = 0.
\]
We consider two examples to lead our discussion
\begin{align*}
    (i) \ \{a,b,-c,-d\} &= \{1,1,-1,-1\} &     (ii) \ \{a,b,-c,-d\} &= \{2,2,-3,-1\}.
\end{align*}
Note that even though both examples satisfy $a+b=c+d$, in $(i)$ there is a smaller set of coefficients that sums to zero, namely $a-c=0$ and $b-d=0$. 
This is not possible for example $(ii)$. We will call $(ii)$ to be of genus 1, and $(i)$ to be of genus 2.
\begin{definition}[Genus]
We say that \eqref{4variable equation} is of \textbf{genus one} if there does not exist a strict subset $T$ of $\{a,b,-c,-d\}$ such that $\sum\limits_{t_i\in T} t_i=0$.
Otherwise we say \eqref{4variable equation} is of \textbf{genus two}. In this case $\{a,b\}=\{c,d\}$.
\end{definition}

\begin{definition}
Consider an equation satisfying \eqref{4variable equation}. 
\begin{itemize}
    \item If \eqref{4variable equation} is of genus one, then a solution $\{x_1, x_2, y_1, y_2 \}$ of \eqref{4variable equation} is said to be \textbf{trivial} if $x_1=x_2=y_1=y_2$. 

    \item If \eqref{4variable equation} is of genus two, a solution is called \textbf{trivial} if $x_1=y_1$ and $x_2 = y_2$.

    \item All other solutions (in any case) are referred to as \textbf{nontrivial solutions}.
\end{itemize}
\end{definition}

\subsection{Proof Techniques and Layout of the Paper}

To prove Theorem \ref{theorem: containment result} we follow the strategies introduced in \cite{LiangandPramanik}.
There the authors construct specific measures to identify nontrivial solutions to specific translation invariant linear equations. 
In this rendition, section \ref{section: measure construction} is dedicated to constructing similar measures that identify nontrivial solutions to equations of the form \eqref{4variable equation}. 
Once that is complete, section \ref{section: main theorem proof} sets about proving theorem \ref{theorem: containment result}. 
The final section describes the construction of a Salem set that is then used to prove theorem \ref{E For avoidance}.

\section{Construction of Measures}\label{section: measure construction}

We construct a class of measures and apply them in the next section to prove theorem \ref{theorem: containment result}. 
To this end, fix a closed set $E\subset [0,1]$ satisfying $\dim_\mathbf{F}(E)>\frac{1}{2}$ and for each choice of coefficients $\{a,b,c,d\}\subset \N$ the goal is to construct a measure that if nonzero signals the existence of nontrivial solutions to \eqref{4variable equation} inside $E$. 
To build up to this measure, we start with a functional that acts on the subspace $\mathcal{V}\subset C([0,1]^3)$ of Schwartz functions whose Fourier transform is compactly supported. 
We note that $\mathcal{V}$ is dense in $C([0,1]^3)$ under the sup norm.
This was also showed in \cite{LiangandPramanik} but we include proof in section \ref{section: denseness} for completeness.

Given any choice $\{a,b,c,d\}\subset \N$, the associated equation \eqref{4variable equation} can be identified with a pair $(s,t)\in (0,1)^2\cap \Q^2$. 
To see this set $S=a+b=c+d$, write equation \eqref{4variable equation} as
\begin{equation}
\frac{a}{S} x_1 + \frac{b}{S}x_2 - \frac{c}{S}y_1 - \frac{d}{S}y_2 = 0 \text{ and define } g(x_1,x_2;y_1,y_2) := \frac{a}{S} x_1 + \frac{b}{S}x_2 - \frac{c}{S}y_1 - \frac{d}{S}y_2.
\end{equation}
We use this pair $(s,t)$ to construct our desired measure; however, in a later lemma, this identification will become useful.

Recall that $E\subset[0,1]$ with $\dim_\mathbf{F}(E)=\beta>1/2$ and let $\mu$ be a probability measure supported on $E$ satisfying 
\[
|\hat{\mu}(\xi)| \le C (1+|\xi|)^{-\beta/2}
\]
for some positive finite constant $C$. 
We mollify this measure $\mu$ with the following approximation to the identity to define a functional $\Lambda_{(s,t)}$.
Take $\psi\in C_c^\infty (\R)$, $\psi \ge 0$, supp$(\psi)\subset[-1,1]$ and $\int\psi = 1$.
For all $\epsilon>0$ define
\begin{equation}
\psi_\epsilon(x) = \frac{1}{\epsilon}\psi\left(\frac{x}{\epsilon}\right) \text{ and } \mu_\epsilon := \mu * \psi_\epsilon.
\end{equation}
Observe that $\mu_\epsilon$ is supported on an $\epsilon$-neighborhood of $E$, denoted by $\mathcal{N}_\epsilon(E)$.

For each pair $(s,t)\in(0,1)^2\cap \Q^2$ define the following limit for functions $f\in \mathcal{V}$ 
\begin{align}\label{epsilon lambda}
\langle \Lambda_{(s,t)} , f \rangle &:= \lim_{\epsilon\rightarrow 0} \langle \Lambda_{(s,t)}^\epsilon , f \rangle \\
&= \lim_{\epsilon \rightarrow 0}\int\limits_{[0,1]^3} f(x_1,x_2,y_1)\mu_\epsilon(x_1)\mu_\epsilon(x_2)\mu_\epsilon(y_1)\mu_\epsilon\left(    \frac{sx_1 +(1-s)x_2-ty_1}{1-t}   \right)dx_1dx_2dy_1. \notag
\end{align}

The following Lemma establishes this limit as well defined. 

\begin{lemma}\label{lemma: Fourier side of the limit}
If $f\in\mathcal{V}$, then for every $(s,t)\in (0,1)^2$, the limit in \eqref{epsilon lambda} is absolutely convergent and equals 
    \begin{equation}\label{Fourier version of config int}
    \langle \Lambda_{(s,t)} , f \rangle = \int\limits_{\R^4}\hat{\mu}(\xi) \prod_{i=1}^3\hat{\mu}(\eta_i) \hat{f} \left(  -\eta_1 -\frac{s}{1-t}\xi , -\eta_2 - \frac{(1-s)}{1-t} \xi , -\eta_3 + \frac{t}{1-t}\xi  \right)d\xi d\eta_1 d\eta_2 d\eta_3. 
    \end{equation}
\end{lemma}

Even further we can find a constant so that the limit above behaves like a functional over $\mathcal{V}$.

\begin{lemma}\label{lemma: constant for linear functional}
For every $(s,t)\in(0,1)^2$, there exists a positive finite constant $C'((s,t))$ depending only on $C$ and $(s,t)$ such that for all $f\in \mathcal{V}$
\begin{equation}\label{Statement towards a linear functional}
\sup_{\epsilon>0}|\langle \Lambda_{(s,t)}^\epsilon , f \rangle| \le C'((s,t)) ||f||_\infty \text{ and so } |\langle \Lambda_{(s,t)} , f \rangle| \le C'((s,t)) ||f||_\infty.
\end{equation}
\end{lemma}

Using Lemma \ref{lemma: Fourier side of the limit} and Lemma \ref{lemma: constant for linear functional} we are equipped to show the existence of the measure we are after, this is the content of the following proposition. 

\begin{proposition}\label{prop: summary of measure properties}
Fix a choice of coefficients $\{a,b,c,d\}\subset \N$ satisfying $a+b=c+d$ and consider the associated equation as in \eqref{4variable equation}. 
Let $E\subset[0,1]$ be a closed set and $\mu$ be a probability measure supported on $E$ satisfying
\[
|\hat{\mu}(\xi)| \le C (1+|\xi|)^{-\beta}
\]
for some finite positive constant $C$ and $\beta>1/4$. 
Then the following conclusions hold:
\begin{enumerate}[label = (\alph*)]
    \item The limit in \eqref{epsilon lambda} holds for all $f\in C([0,1]^3)$. Furthermore $\langle \Lambda_{(s,t)} , f \rangle $ can be identified with a nonnegative bounded linear functional on $C([0,1]^3)$ obeying
    \[
    \left|\langle \Lambda_{(s,t)} , f \rangle \right| \le C'(s,t) ||f||_\infty.
    \]
    In particular, $\Lambda_{(x,t)}$ can be identified as integration against a nonnegative Borel measure which we also denote by $\Lambda_{(s,t)}$.

    \item If $\Lambda_{(s,t)}$ is nontrivial, then supp$(\Lambda_{(s,t)}) \subset \mathcal{S}$ and assigns measure zero to the trivial solutions $\mathcal{T}$, which is given by 
    \[
    \mathcal{T}= \left\{(x_1,x_2,y_1,y_2)\in \mathcal{S}(E,(s,1-s,t,1-t)) :  (x_1,x_2)=(y_1,y_2) \ \text{or} \ (x_1,x_2)=(y_2,y_1) \right\}.
    \]
\end{enumerate}
\end{proposition}

\subsection{Proof of Lemma \ref{lemma: Fourier side of the limit}}
\begin{proof}
The proof proceeds by establishing absolute convergence of \eqref{epsilon lambda} then moves our expression \eqref{epsilon lambda} to the frequency side where we conclude with the dominated convergence theorem.

To establish absolute convergence, since $f\in\mathcal{V}$, the expression defining \eqref{epsilon lambda} is a continuous function integrated over a compact domain. 

To establish \eqref{Fourier version of config int}, absolute convergence allows manipulation as desired. 
Fourier inversion provides 
\begin{align*}
    \langle \Lambda_{(s,t)}^\epsilon,f \rangle &= \int\limits_{[0,1]^3}f(x_1,x_2,y_1)\int\limits_{\R^4}\prod_{i=1}^3\widehat{\mu_\epsilon} (\eta_i) \widehat{\mu_\epsilon}(\xi)e^{2\pi i \xi \left(\frac{sx_1 +(1-s)x_2-ty_1}{(1-t)} \right)}d\eta_1d\eta_2d\eta_3d\xi dx_1dx_2dy_1 \\
    &\le \int\limits_{\R^4} \prod_{i=1}^3\widehat{\mu_\epsilon}(\eta_i) \widehat{\mu_\epsilon}(\xi) \\
    & \quad \times \int\limits_{\R^3} f(x_1,x_2,y_1) e^{2\pi i(x_1,x_2,y_1)\cdot \left(\eta_1+\xi\frac{s}{1-t}, \eta_2+\xi \frac{1-s}{1-t}, \eta_3 - \xi \frac{t}{1-t}\right)}dx_1dx_2dy_1 d\xi d\eta_1d\eta_2 d\eta_3 \\
    &\le  \int\limits_{\R^4} \prod_{i=1}^3\widehat{\mu_\epsilon}(\eta_i) \widehat{\mu_\epsilon}(\xi) \hat{f} \left( -\eta_1-\frac{s}{1-t}\xi , -\eta_2 -\frac{1-s}{1-t}\xi , -\eta_3 + \frac{t}{1-t} \xi    \right)d\xi d\eta_1d\eta_2 d\eta_3.
\end{align*}
For all $\xi \in \R$ as $\epsilon\rightarrow 0$, we have $\hat{\mu}_\epsilon (\xi)\rightarrow \hat{\mu}(\xi)$; hence, the above integral converges pointwise to \eqref{Fourier version of config int}.
By the dominated convergence theorem, we can conclude if we show 
\begin{equation}\label{ft integral needed finite}
\int\limits_{\R^4} \left|\prod_{i=1}^3\widehat{\mu}(\eta_i) \widehat{\mu}(\xi) \hat{f} \left( -\eta_1-\frac{s}{1-t}\xi , -\eta_2 -\frac{1-s}{1-t}\xi , -\eta_3 + \frac{t}{1-t} \xi    \right)  \right|  d\xi d\eta_1d\eta_2 d\eta_3 <\infty.
\end{equation}

Toward showing \eqref{ft integral needed finite}, we take advantage of the compact support of $\hat{f}$ by first rewriting as
\begin{align*}
    \eqref{ft integral needed finite} &= \int\limits_\R |\hat{\mu}(\xi)| \mathcal{D}_\xi d\xi, \quad \quad  \text{ where } \\
    \mathcal{D}_\xi &:= \int\limits_{\R^3} \left|\prod_{i=1}^3\widehat{\mu}(\eta_i)\hat{f} \left( -\eta_1-\frac{s}{1-t}\xi , -\eta_2 -\frac{1-s}{1-t}\xi , -\eta_3 + \frac{t}{1-t} \xi    \right)\right|d\eta_1 d\eta_2 d\eta_3. 
\end{align*}
We claim that $\mathcal{D}_\xi$ obeys the following bound:
\begin{equation}\label{Dxi bound for finite fourier integral}
\mathcal{D}_\xi \le ||\hat{f}||_\infty C^3(\beta) (6R)^3  \left(1+ c |\xi|\right)^{-3\beta}
\end{equation}
where $R>0$ is a fixed constant given by the support of $\hat{f}$. 
Assuming \eqref{Dxi bound for finite fourier integral}, we conclude the proof as follows
\begin{align*}
    \int\limits_\R \left|\widehat{\mu} (\xi)\right|\mathcal{D}_\xi d\xi&\le ||\hat{f}||_\infty C^4(\beta) (6R)^3  \int\limits_\R (1+|\xi|)^{-\beta} \left(1+c |\xi|\right)^{-3\beta} d\xi \\
    &\le ||\hat{f}||_\infty C^4(\beta) (6R)^3  \int\limits_\R (1+ \min\{1,c\}|\xi|)^{-4\beta} d\xi <\infty
\end{align*}
because $4\beta>1$ by assumption.

To prove our claim, we show the desired bound for a given $\xi \in \R$. 
Recall that some fixed $R>0$ exists such that supp$(\hat{f}) \subset[-R,R]^3$ so that for a given $\xi$, $\mathcal{D}_\xi$ has the following domain of integration 
\[
D_\xi =\left\{( \eta_1, \eta_2 ,\eta_3)\in \R^3 : \left|-\eta_1-\frac{s}{1-t}\xi\right|\le R , \left|-\eta_2 -\frac{1-s}{1-t}\xi\right|\le R , \left|-\eta_3 + \frac{t}{1-t} \xi\right|\le R     \right\}
\]
which is contained in the following domain
\[
D_\xi \subset \left[  -R -\frac{s}{1-t}\xi , R - \frac{s}{1-t}\xi \right] \times 
\left[ -R -\frac{1-s}{1-t}\xi , R - \frac{1-s}{1-t}\xi    \right]\times 
\left[   -R + \frac{t}{1-t}\xi , R +\frac{t}{1-t}\xi \right]
\]
for each $\eta_1$, $\eta_2$ and $\eta_3$, respectively. 

We approximate $\mathcal{D}_\xi$ by 
\begin{equation}\label{split into three independent integrals}
\mathcal{D_\xi} \le ||\hat{f}||_\infty \int\limits_{-R-\frac{s}{1-t}\xi}^{R-\frac{s}{1-t}\xi}|\hat{\mu}(\eta_1)|d\eta_1 \int\limits_{-R-\frac{1-s}{1-t}\xi}^{R-\frac{1-s}{1-t}\xi}|\hat{\mu}(\eta_2)|d\eta_2 \int\limits_{-R+\frac{t}{1-t}\xi}^{R+\frac{t}{1-t}\xi}|\hat{\mu}(\eta_3)|d\eta_3.
\end{equation}
Using the Fourier decay of $\mu$ allows us to estimate for each $\eta_i$, $i=1,2,3$ and $c_1= \frac{s}{1-t}$, $c_2=\frac{1-s}{1-t}$ and $c_3=\frac{t}{1-t}$,
\begin{align}\label{bound on eta integrals}
\int\limits_{-R \pm c_i\xi}^{R \pm c_i\xi} |\hat{\mu}(\eta_i)|d\eta_i &\le C(\beta)
\begin{cases}
    \int\limits_{-3R}^{3R} d\eta_i & \text{ if } c_i|\xi|\le 2R  \\
    \int\limits_{-R \pm c_i\xi}^{R \pm c_i\xi} (1+c_i|\eta_i|)^{-\beta}d\eta_i & \text{ if } c_i|\xi|>2R
\end{cases} \notag \\
&\le C(\beta) (6R) (1+c_i|\xi|)^{-\beta} 
\end{align}
Plugging \eqref{bound on eta integrals} into \eqref{split into three independent integrals} for each $i$ results in 
\begin{equation}
\mathcal{D}_\xi \le ||\hat{f}||_\infty C^3(\beta) (6R)^3 (1+c_1|\xi|)^{-\beta}(1+c_2|\xi|)^{-\beta}(1+c_3|\xi|)^{-\beta}.
\end{equation}
We conclude by taking $c=\max\{c_1,c_2,c_3\}$, which allows for our desired bound
\[
\mathcal{D}_\xi \le ||\hat{f}||_\infty C^3(\beta) (6R)^3 (1+c |\xi|)^{-3\beta}. 
\]
\end{proof}

\subsection{Proof of Lemma \ref{lemma: constant for linear functional}}

\begin{proof}
Since the right side of \eqref{Statement towards a linear functional} follows from the left side, we only show the left side. 
We do this by again transferring to the frequency side of our expression and approximate. 

Again, we take advantage of absolute convergence and apply Fourier inversion as follows
\begin{align*}
    \left|\langle \Lambda_{(s,t)}^\epsilon , f \rangle\right| &= \left|\int\limits_{[0,1]^3} f(x_1,x_2,y_1)\mu_\epsilon(x_1)\mu_\epsilon(x_2)\mu_\epsilon(y_2)\mu_\epsilon\left( \frac{sx_1+(1-s)x_2-ty_1}{1-t} \right)dx_1dx_2dy_1 \right| \\
    &\le ||f||_\infty \int\limits_{[0,1]^3} \mu_\epsilon(x_1)\mu_\epsilon(x_2)\mu_\epsilon(y_2)\mu_\epsilon\left(\frac{sx_1+(1-s)x_2-ty_1}{1-t} \right)dx_1dx_2dy_1  \\
    &\le ||f||_\infty \int\limits_{\R^3} \mu_\epsilon(x_1)\mu_\epsilon(x_2)\mu_\epsilon(y_2)\int\limits_\R \widehat{\mu_\epsilon}(\xi)e^{2\pi i\xi \left(\frac{sx_1+(1-s)x_2-ty_1}{1-t} \right)}d\xi dx_1dx_2dy_1 \\
    &\le ||f||_\infty \int\limits_\R\widehat{\mu_\epsilon}(\xi) \widehat{\mu_\epsilon}\left(\frac{-s}{1-t}\xi\right)\widehat{\mu_\epsilon}\left(\frac{-(1-s)}{1-t}\xi\right)\widehat{\mu_\epsilon}\left(\frac{t}{1-t}\xi \right) d\xi  \\
    &\le ||f||_\infty \int\limits_\R\widehat{\mu}(\xi) \widehat{\mu}\left(\frac{-s}{1-t}\xi\right)\widehat{\mu}\left(\frac{-(1-s)}{1-t}\xi\right)\widehat{\mu}\left(\frac{t}{1-t}\xi \right) d\xi  \\
    &\le ||f||_\infty C^4(\beta)\int\limits_\R \left(1+\min\left\{1,\frac{s}{1-t} , \frac{1-s}{1-t},\frac{t}{1-t}\right\} |\xi|\right)^{-4\beta}d\xi < \infty. 
\end{align*}
The assumption $4\beta>1$ allows for convergence of the above integral, the constant is explicitly
\begin{equation}\label{the constant we integrate over}
    C'((s,t)) = C^4(\beta)\int\limits_\R (1+|\xi|)^{-\beta}\left(1+\frac{s}{1-t} |\xi|\right)^{-\beta}\left(1+ \frac{1-s}{1-t} |\xi|\right)^{-\beta}\left(1+ \frac{t}{1-t}|\xi|\right)^{-\beta}d\xi.
\end{equation}
\end{proof}

\subsection{Proof of Proposition \ref{prop: summary of measure properties} Part $(a)$}\label{section: functional piece of main prop}

\begin{proof}
To show that \eqref{epsilon lambda} holds for all $f\in C([0,1]^3)$, we approximate it via functions $g\in \mathcal{V}$.

Given some $f\in C([0,1]^3)$, fix $g\in \mathcal{V}$ such that $\norm{f-g}_\infty< \frac{\delta}{C'(s,t)}$ for $\delta>0$; this is possible by Lemma \ref{lemma:denseness of V}. 
Because $g\in \mathcal{V}$, $|\langle \Lambda_{(s,t)}^\epsilon -\Lambda_{(s,t)} , g \rangle|\le \delta $ for all $\epsilon$ small enough. 
Now, we estimate with \eqref{epsilon lambda}
\begin{align*}
    |\langle \Lambda_{(s,t)}^\epsilon -\Lambda_{(s,t)} , f \rangle| &\le |\langle \Lambda_{(s,t)}^\epsilon -\Lambda_{(s,t)} , g \rangle| +|\langle \Lambda_{(s,t)}^\epsilon -\Lambda_{(s,t)} , f- g \rangle| \\
    &\le |\langle \Lambda_{(s,t)}^\epsilon -\Lambda_{(s,t)} , g \rangle| +|\langle \Lambda_{(s,t)}^\epsilon , f- g \rangle| +|\langle \Lambda_{(s,t)}, f- g \rangle| \\
    &\le \delta +2 C'(s,t) \norm{f-g}_\infty  \le 3\delta
\end{align*}
which holds independent of $\epsilon \rightarrow 0$.

Establishing \eqref{epsilon lambda} as a functional over $C([0,1]^3)$ leads to the use of the Riesz representation theorem and allows us to conclude. 
To establish this, we first observe that the definition of $\Lambda_{(s,t)}$ allows \eqref{Statement towards a linear functional} to hold for all $f\in C([0,1]^3)$ so we show that $\{\langle\Lambda_{(s,t)}, f_n\rangle\}$ is a Cauchy sequence and converges to a limit denoted by $\langle \Lambda_{(s,t)},f\rangle$.

By Lemma \ref{lemma:denseness of V}, there exists a sequence of functions $\{f_n\}\subset \mathcal{V}$ such that $f_n\rightarrow f$ uniformly. 
For $n,m\in \N$, the function $f_n-f_m\in \mathcal{V}$, so that by Lemma \ref{lemma: constant for linear functional}, we obtain
\[
    |\langle  \Lambda_{(s,t)}^\epsilon , f_n-f_m   \rangle| \le C'(s,t) \norm{f_n-f_m}_\infty 
\]
and consequently 
\[
    |\langle  \Lambda_{(s,t)} , f_n-f_m   \rangle| \le C'(s,t) \norm{f_n-f_m}_\infty \rightarrow 0. 
\]
Thus, $\{\langle\Lambda_{(s,t)}, f_n\rangle\}$ is a Cauchy sequence and hence converges weakly to a limit denoted by $\langle \Lambda_{(s,t)}, f \rangle$. 
That the limit is independent of the approximating sequence, take another sequence $\{g_n\}\subset \mathcal{V}$ that converges to $f$ and observe that Lemma \ref{lemma: constant for linear functional} again gives
\[
    |\langle  \Lambda_{(s,t)} , f_n-g_n   \rangle| \le C'(s,t) \norm{f_n-g_n}_\infty\rightarrow 0.
\]
\end{proof}

\subsection{Proof of Proposition \ref{prop: summary of measure properties} Part $(b)$}\label{section: support condition proof}

\begin{proof}
We show the support conditions in order as they appear. 
For supp$(\Lambda_{(s,t)})\subset\mathcal{S}$, we show that any function $f\in C([0,1]^3)$ that is supported on a cube outside of $\mathcal{S}$. 
To show that this measure assigns a value of zero to $\mathcal{T}$, we consider when a specific integral vanishes. 

If assumed nontrivial, to show supp$(\Lambda_{(s,t)})\subset\mathcal{S}$, we proceed by showing that for any $w\in [0,1]^4\setminus \mathcal{S}(E,(s,1-s,t,1-t))$, we can find an $r>0$ such that for a closed cube of length $r$ centered around $u$,
\[
Q(w,r)\subset[0,1]^4\setminus \mathcal{S}_{(s,t)}(E) \
\text{and} \
\langle \Lambda_{(s,t)} , f\rangle =0 
\]
for any $f\in C([0,1]^3)$ with supp$(f)\subset Q(w,r)$. 

Fix $w\in  [0,1]^4\setminus \mathcal{S}$ such that at least one index $1\le i\le 4$ satisfies $w_i\not\in E$. 
Since $E$ is closed, some positive $r$ must exist so that for all $\epsilon<r$
\[
\text{dist}(w_i,E ) >2r \  \text{ and } \  \text{dist}(w_i, \mathcal{N}_\epsilon[E])>r. 
\]
Consequently $\overline{B(w_i,\frac{r}{2})}\cap \mathcal{N}_\epsilon[E]=\emptyset$ and $\mu_\epsilon\equiv0$ in the associated $w_i$ variable, without loss of generality, assume this happens for the variable $x$.
Furthermore, for any $f$ supported on $Q(w,r)$ and $\epsilon <r$, we must have  $f(x_1,x_2,y_1,y_2)\mu_\epsilon(x_1)=0$. 
Combining this with \eqref{epsilon lambda} results in $\langle \Lambda_{(s,t)}^\epsilon, f\rangle=0$ for all $\epsilon<r$ and hence $\langle \Lambda_{(s,t)},f\rangle=0$. 
We conclude that the measure $\Lambda_{(s,t)}$ is supported in $\mathcal{S}$. 


Moving toward the second statement about $\mathcal{T}$, we first observe that $\mathcal{T}$ depends on the coefficients of a given equation, but in all cases at least one pair of variables is equal. 
Therefore, it suffices to show that our measure $\Lambda_{(s,t)}$ vanishes on Schwartz functions supported on at least one of the following sets
\[
\{(x_1,x_2,y_1,y_2)\in [0,1]^4 \: \ x_1=y_1 \} \text{ or } \{(x_1,x_2,y_1,y_2)\in [0,1]^4 \: \ x_1=y_2 \} 
\]
and whose Fourier transform is compactly supported. 
Let $\chi : \R \rightarrow [0,\infty)$ be a nonnegative Schwartz function such that $\chi\ge 1$ on $[-1,1]$ and supp$( \hat{\chi})\subset [-R,R]$ for some positive finite $R$. 
The function we utilize is obtained by setting $f^\delta (x_1,x_2,y_1) = \chi(\delta^{-1}(x_1-y_1))$ and observing that $\Lambda_{(s,t)}(\{ (x_1,x_2,y_1,y_2)\in [0,1]^4 : x_1=y_1 \}) \le \langle   \Lambda_{(s,t)}, f^\delta    \rangle$; hence, it suffices to show that 
\[
\langle   \Lambda_{(s,t)}, f^\delta    \rangle \rightarrow 0 \text{ as } \delta \rightarrow 0 .
\]
We apply \eqref{epsilon lambda} to $f^\delta$, and for $\epsilon>0$, we use Fourier inversion to estimate
\begin{align*}
    \langle \Lambda_{(s,t)}^\epsilon , f^\delta \rangle &= \int\limits_{\R^3} \mu_\epsilon(x_1)\mu_\epsilon(x_2)\mu_\epsilon(y_1)f^\delta(x_1,x_2,y_1) \mu_\epsilon\left( \frac{sx_1+(1-s)x_2-ty_1}{1-t}  \right)dx_1dx_2dy_1\\
    &= \int\limits_{\R^2} \widehat{\mu_\epsilon}(\eta)\widehat{\mu_\epsilon}\left(-\frac{1-s}{1-t}\eta\right)\widehat{\mu_\epsilon}\left(    -\left(\frac{s}{1-t}|\eta|+\xi\right)\right)\widehat{\mu_\epsilon}\left(\frac{t}{1-t}\eta + \xi\right)\hat{\chi}(\delta^{-1}\xi) d\xi d\eta  \\ 
    &\le \int\limits_{\R^2} \left| \hat{\mu}(\eta)\hat{\mu}\left(-\frac{1-s}{1-t}\eta\right)\hat{\mu}\left(    -\left(\frac{s}{1-t}|\eta|+\xi\right)\right)\hat{\mu}\left(\frac{t}{1-t}\eta + \xi\right)\hat{\chi}(\delta^{-1}\xi)      \right| d\xi d\eta .
\end{align*}
Since this bound is independent of $\epsilon$, it provides a bound for $\langle \Lambda_{(s,t)} , f^\delta \rangle$. 
This is the integral we wish to show vanishes, so we introduce some notation for convenience
\[
\mathcal{J}(\delta) := \int\limits_{\R^2} \left| \hat{\mu}(\eta)\hat{\mu}\left(-\frac{1-s}{1-t}\eta\right)\hat{\mu}\left(    -\left(\frac{s}{1-t}|\eta|+\xi\right)\right)\hat{\mu}\left(\frac{t}{1-t}\eta + \xi\right)\hat{\chi}(\delta^{-1}\xi)\right| d\xi d\eta .
\]

Moving toward our next step, we begin our approximation of $\mathcal{J}$ by writing out definitions of $\hat{\chi}$ and exploiting that it has compact support 
\begin{align}\label{up until cs inequality}
    \mathcal{J}(\delta) &= \int\limits_\R \left|\hat{\mu}(\eta) \hat{\mu}\left(   -\frac{1-s}{1-t} \eta  \right) \right| \int\limits_\R \left|  \hat{\mu}\left(-\left(\frac{s}{1-t}\eta+\xi\right)\right)\hat{\mu}\left(\frac{t}{1-t}\eta + \xi\right)\hat{\chi}(\delta^{-1}\xi) \right| d\xi d\eta   \notag \\
    &\le \norm{\hat{\chi}}_\infty \delta  \int\limits_\R \left|\hat{\mu}(\eta) \hat{\mu}\left(-\frac{1-s}{1-t} \eta  \right) \right| \notag \\
    & \quad \times \int\limits_\R \left|  \hat{\mu}\left(\frac{-s}{1-t}\eta-\xi\right)\mathbf{1}_{\left[\frac{-R}{\delta} , \frac{R}{\delta}\right]}^2(\xi) \hat{\mu}\left(\frac{t}{1-t}\eta + \xi\right) \right| d\xi d\eta \notag \\
    &\lesssim \delta \int\limits_\R ( 1+|\eta|)^{-\beta} \left(  1 + \frac{1-s}{1-t} |\eta|  \right)^{-\beta} \notag \\ 
    & \quad \times \int\limits_\R \left(1+\left|\frac{s}{1-t}\eta+\xi\right|\right)^{-\beta} \mathbf{1}_{\left[\frac{-R}{\delta} , \frac{R}{\delta}\right]}^2 (\xi)  \left(1+\left|\frac{t}{1-t}\eta + \xi\right|\right)^{-\beta} d\xi d\eta \notag \\
    &\lesssim \delta  \int\limits_\R ( 1+|\eta|)^{-\beta} \left(  1 + \frac{1-s}{1-t} |\eta|  \right)^{-\beta} F_{\frac{s}{1-t}}(\eta) F_{\frac{t}{1-t}}(\eta) d\eta 
\end{align}
where 
\begin{multline*}
F_{\frac{s}{1-t}}(\eta) = \left[\int\limits_{-\frac{R}{\delta}}^{\frac{R}{\delta}} \left| \left(1+\left|\frac{s}{1-t}\eta+\xi\right|\right)\right|^{-2\beta}d\xi    \right]^{\frac{1}{2}} \text{ and }\\
F_{\frac{t}{1-t}}(\eta) = \left[\int\limits_{-\frac{R}{\delta}}^{\frac{R}{\delta}}\left( 1+  \left|\frac{t}{1-t}\eta + \xi\right)\right|^{-2\beta} d\xi   \right]^{\frac{1}{2}}.
\end{multline*}
Applying the Cauchy-Schwarz inequality to \eqref{up until cs inequality} results in 
\begin{equation}\label{J delta with norms}
\mathcal{J}(\delta) \le \delta \norm{(1+|\eta|)^{-\beta}\left(1+\frac{1-s}{1-t}|\eta|\right)^{-\beta}}_2  \norm{F_{\frac{s}{1-t}}(\eta)F_{\frac{1-s}{1-t}}(\eta)}_2. 
\end{equation}
The first term in the product leads to 
\begin{equation}\label{Bound that comes out of holder pnorm}
C_2=\norm{(1+|\eta|)^{-\beta}\left(1+\frac{1-s}{1-t}|\eta|\right)^{-\beta}}_2^2 \le \int\limits_\R (1+c|\eta|)^{-4\beta } d\eta <\infty 
\end{equation}
where $c= \max \{1, \frac{1-s}{1-t}\}$.
Recall that $4\beta>1$; hence, the integral converges and, in particular, is independent of $\delta$. 
We claim that $\norm{F_{\frac{s}{1-t}} F_{\frac{t}{1-t}}}_2$ obeys the following bound
\begin{multline}\label{FF q norm bound}
\norm{F_{\frac{s}{1-t}}(\eta) F_{\frac{t}{1-t}}(\eta)}_2 \le \frac{4R}{\delta^{1/2}} \text{max}\left\{1 , \left(1+\frac{3R}{\delta}\right)^{1-2\beta}\log\left(1+\frac{3R}{\delta}\right) \right\}  \\ 
+ \frac{2R}{(8\beta -1)^{\frac{1}{2}} \delta} \left(1+\gamma \left|\frac{2R}{\delta}\right|\right)^{\frac{1}{2} -2\beta} .
\end{multline}

Assuming that \eqref{FF q norm bound} holds and plugging this along with \eqref{Bound that comes out of holder pnorm} into \eqref{J delta with norms}, the argument ends by 
\begin{align*}
    \mathcal{J}(\delta) &\lesssim \delta  C_2 \frac{4R}{\delta^{1/q}} \text{max}\left\{1 , \left(1+\frac{3R}{\delta}\right)^{1-2\beta}\log\left(1+\frac{3R}{\delta}\right) \right\} + \frac{2R}{(8\beta -1)^{\frac{1}{q}} \delta} \left(1+\gamma \left|\frac{2R}{\delta}\right|\right)^{\frac{1}{q} -2\beta} \\
    &\lesssim \delta \begin{cases}
        \frac{4R }{\delta^{\frac{1}{q}}} + \frac{3R}{(8\beta -1)^{\frac{1}{q}} \delta} (1+\gamma|\frac{2R}{\delta}|)^{\frac{1}{q} - 2\beta} & \text{ if } \frac{1}{2} < \beta \le 1 \\
        \frac{4R}{\delta } (1+ \frac{3R}{\delta })^{1-2\beta}\log(1+\frac{3R}{\delta}) + \frac{2R}{\delta (8\beta -1)^{\frac{1}{q}}}  (1+ \gamma |\frac{2R}{\delta}|)^{\frac{1}{q} - 2\beta} & \frac{1}{4} < \beta \le \frac{1}{2}
    \end{cases} \\
    &\le \begin{cases}
        4R \delta^{1-\frac{1}{q}} + \frac{3R}{(2q\beta -1)^{\frac{1}{q}}} (1+\gamma |\frac{2R}{\delta}|)^{\frac{1}{q} - 2\beta } & \text{ if } \frac{1}{2} < \beta \le 1 \\ 
        4R \delta^{2\beta -\frac{1}{q}} (\delta +3R)^{1-2\beta} \log (1+\frac{3R}{\delta}) + 
        \frac{2R}{(8\beta -1)^{\frac{1}{q}}} 
        \delta^{2\beta -1 }(\delta+\gamma 2R)^{\frac{1}{q} - 2\beta}  & \frac{1}{4} < \beta \le \frac{1}{2}
    \end{cases}.
\end{align*}
In any case, $\delta$ is raised to a positive power. 
Hence, $\mathcal{J}(\delta)\rightarrow 0$ as $\delta\rightarrow 0$ as desired.

All we must do is now show that \eqref{FF q norm bound} holds. 
We first obtain an estimate for $F_{\frac{s}{1-t}}$ and $F_{\frac{t}{1-t}}$. 
Both terms obey a similar bound as follows, but we only show one
\begin{align*}
    [F_{\frac{s}{1-t}}(\eta)]^2 &= \int\limits_{\frac{-R}{\delta}}^{\frac{R}{\delta}} \left(1+\left| \frac{s}{1-t}\eta + \xi   \right|\right)^{-2\beta} d\xi  \\
    &\le \begin{cases}
        \int\limits_{\frac{-3R}{\delta}}^{\frac{3R}{\delta}} (1+|w|)^{-2\beta}dw & \text{ if } \frac{s}{1-t}|\eta|\le \frac{2R}{\delta} \\
        \int\limits_{\frac{-R}{\delta}}^{\frac{R}{\delta}} \left(1+\frac{s}{2(1-t)}|\eta|\right)^{-2\beta} d\xi & \text{ if } \frac{s}{1-t}|\eta| > \frac{2R}{\delta}
    \end{cases} \\
    &\le \begin{cases}
        \text{max}\{1 , (1+\frac{3R}{\delta})^{1-2\beta}\log(1+\frac{3R}{\delta}) \} & \text{ if } \frac{s}{1-t}|\eta|\le \frac{2R}{\delta} \\
        \frac{2R}{\delta } (1+\frac{s}{2(1-t)}|\eta|)^{-2\beta} & \text{ if } \frac{s}{1-t}|\eta| > \frac{2R}{\delta}
    \end{cases}.
\end{align*}
Using the above bound, we approximate $\norm{F_{\frac{s}{1-t}} F_{\frac{t}{1-t}}}_2$ as 
\begin{align*}
    \norm{F_{\frac{s}{1-t}}(\eta) F_{\frac{t}{1-t}}(\eta)}_2^2 &= \int\limits_{|\eta | \le \frac{2R}{\delta}} [F_{\frac{s}{1-t}}(\eta) F_{\frac{t}{1-t}}(\eta)]^2 d\eta + \int\limits_{\frac{2R}{\delta}< |\eta|} [F_{\frac{s}{1-t}}(\eta) F_{\frac{t}{1-t}}(\eta)]^2 d\eta  \\
    &\le \frac{4R}{\delta} \left[\text{max}\left\{1 , \left(1+\frac{3R}{\delta}\right)^{1-2\beta}\log\left(1+\frac{3R}{\delta}\right) \right\}\right]^4 \\
    & \quad + \frac{2R}{(8\beta -1)^{\frac{1}{4}}\delta} \left(1 + \gamma \left|\frac{2R}{\delta}\right|\right)^{\frac{1}{4}-2\beta} 
\end{align*}
where $\gamma =$min$\{\frac{s}{1-t}, \frac{t}{1-t}\}$. 
Finally,
\begin{multline*}
\norm{F_{\frac{s}{1-t}}(\eta) F_{\frac{t}{1-t}}(\eta)}_2 \le \frac{4R}{\delta^{1/2}} \text{max}\left\{1 , \left(1+\frac{3R}{\delta}\right)^{1-2\beta}\log\left(1+\frac{3R}{\delta}\right) \right\} \\
+ \frac{2R}{(8\beta -1)^{\frac{1}{2}} \delta} \left(1+\gamma \left|\frac{2R}{\delta}\right|\right)^{\frac{1}{2} -2\beta} .
\end{multline*}
\end{proof}

\subsection{Denseness of $\mathcal{V}$ in $C([0,1]^3)$}\label{section: denseness}

Note that our application of this Lemma is for $v=3$. 

\begin{lemma}\label{lemma:denseness of V}
Let $v\in \N$. 
The class of Schwartz functions $\mathcal{V}$ that map $\R^v$ to $\C$ with compact Fourier support is dense in $C([0,1]^v)$.
Furthermore, for any $f\in C([0,1]^v)$, there exists a sequence of functions in $\mathcal{V}$ that is uniformly convergent to $f$ on $[0,1]^v$.
\end{lemma}

\begin{proof}
Our proof strategy is to show that continuous functions on $\R^v$ of compact support can be approximated uniformly on compact sets with functions in $\mathcal{V}$. 
This gives us our desired result for any $f\in C([0,1]^v)$ because any such $f$ has a continuous extension to some $f_c\in C_C(\R^v)$. 

Let $\epsilon >0$, $\phi\in \mathcal{V}$ so that $\hat{\phi}(0)=1$ and $\hat{\phi}$ is smooth and compactly supported. 
We define 
\[
\phi_\epsilon:= \epsilon^{-1} \phi\left(\frac{x}{\epsilon} \right)
\]
which is an approximation to the identity as $\epsilon\rightarrow 0$. 
Observe that $\{f_c * \phi_\epsilon \}$ is a family of functions that converges to $f_c$ as $\epsilon\rightarrow 0$. 
Furthermore, $\widehat{f_c * \phi_\epsilon} = \widehat{f_c}\widehat{\phi_\epsilon}$, which is compactly supported because $\widehat\phi_\epsilon$ has compact support. 
Hence, $\{f_c * \phi_\epsilon \}$ is a family of functions contained in $\mathcal{V}$. 

Finally, because $f_c\in L^\infty$, $\{f_c * \phi_\epsilon \}$ converges uniformly to $f_c$ on compact sets. 
In particular, on $[0,1]^v$, we obtain that $\{f_c * \phi_\epsilon \}$ converges uniformly to $f_c\equiv f$ as desired. 
\end{proof}

\section{Proof of Theorem \ref{theorem: containment result}}\label{section: main theorem proof}

\subsection{Proof of Theorem \ref{theorem: containment result}}

\begin{proof}
The proof proceeds by contradiction; that is, we assume $E$ contains no nontrivial solutions to any equation of the form $f$. 
Using the given coefficients $\{a,b,c,d\}$, we can set 
\[
s= \frac{a}{a+b} \text{ and } t= \frac{c}{c+d}.
\]
Thus, our assumption is equivalent to $E$ avoiding solutions to any equation of the form 
\[
f(x_1,x_2;y_1,y_2) = sx_1 +(1-s)x_2 - ty_1 - (1-t)y_2 \text{ where } (s,t) \in (0,1)^2 \cap \Q^2.
\]
The contradiction will follow by first showing that this assumption dictates all the measures constructed in the previous section must be equivalent to zero, but by a computation using definitions, we show that this cannot happen.

Take $f\in \mathcal{V}$ and any pair $(s,t)\in(0,1)^2$ and then define the following function 
\begin{align*}
    F^{[f]}(s,t)&:= \langle \Lambda_{(s,t)},f \rangle \\
    & = \int\limits_{[0,1]^4}\hat{\mu}(\xi) \prod_{i=1}^3\hat{\mu}(\eta_i) \hat{f} \left(  -\eta_1 -\frac{s}{1-t}\xi , -\eta_2 - \frac{(1-s)}{1-t} \xi , -\eta_3 + \frac{t}{1-t}\xi  \right)d\xi d\eta_1 d\eta_2 d\eta_3. 
\end{align*}
Note that $(s,t)\mapsto F^{[f]}(s,t)$ is a continuous function with proof in the following subsection. 
By assuming that $E$ avoids all nontrivial solutions to $f$, it must be that $\Lambda_{(s,t)}\equiv 0$ for all $(s,t)\in (0,1)^2\cap \Q^2$ and that $F^{[f]}(s,t)=0$ for every $(s,t)\in(0,1)^2\cap \Q^2$. 
Furthermore, since this function is continuous we conclude that $F^{[f]}(s,t)=0$ for all $(s,t)\in(0,1)^2$ and all $f\in \mathcal{V}$.
Finally, since $\mathcal{V}$ is dense in $C([0,1]^3)$ with the sup norm, \eqref{epsilon lambda} establishes for all $(s,t)\in(0,1)^2$ and all $f\in C([0,1]^3)$ that $\langle\Lambda_{(s,t)},f\rangle =0$.

Now, we compute the mass of $\mu_\epsilon$ for $\epsilon>0$. 
Note that in what follows, we use the change of variables $v\mapsto s$ with $v= [sx_1 +(1-s)x_2-ty_1]/[1-t]$ in the first inequality 
\begin{align*}
    1&= \left[\int\limits_0^1 \mu_\epsilon(z)dz  \right]^4  = \int\limits_{[0,1]^4} \mu_\epsilon(x_1)\mu_\epsilon(x_2)\mu_\epsilon(y_1)\mu_\epsilon(y_2) dy_2 dx_1dx_2dy_1 \\
    &=4! \int\limits_{\{(x_1,x_2,y_1,y_2)\in[0,1]^4\ : \ x_1<y_1<y_2<x_2\}}  \mu_\epsilon(x_1)\mu_\epsilon(x_2)\mu_\epsilon(y_1)\mu_\epsilon(y_2)dy_2dx_1dx_2dy_1 \\
    &\le 4! \int\limits_{\{(x_1,x_2,y_1)\in[0,1]^3\ :\ x_1<y_1<x_2\}} \int\limits_0^1 \mu_\epsilon(x_1)\mu_\epsilon(x_2)\mu_\epsilon(y_1)\mu_\epsilon\left( \frac{sx_1+(1-s)x_2-ty_1}{1-t}\right)  \left|  \frac{x_1-x_2}{1-t}  \right| ds dx_1dx_2dy_1 \\
    &\le \frac{4!}{1-t}\int\limits_0^1 \int\limits_{[0,1]^3} \mu_\epsilon(x_1)\mu_\epsilon(x_2)\mu_\epsilon(y_1)\mu_\epsilon\left( \frac{sx_1+(1-s)x_2-ty_1}{1-t}\right)  \left|  x_1-x_2  \right|  dx_1dx_2dy_1  ds\\
    &= \frac{4!}{1-t} \int\limits_0^1 \langle \Lambda_{(s,t)}^\epsilon,f_0(x_1,x_2,y_1) \rangle ds
\end{align*}
where $f_0(x_1,x_2,y_1)= |x_1-x_2|$. 
Note that $f_0$ is a continuous function of three variables but is constant in one of them.

The above manipulations lead us to 
\[
1\le \frac{4!}{1-t} \int\limits_0^1 \langle \Lambda_{(s,t)}^\epsilon,|x_1-x_2| \rangle ds
\]
taking a limit as $\epsilon \rightarrow 0$ provides 
\[
1\le\lim_{\epsilon \rightarrow 0} \frac{4!}{1-t} \int\limits_0^1 \langle \Lambda_{(s,t)}^\epsilon, |x_1-x_2| \rangle ds 
\]
To compute the above limit, we wish to invoke the dominated convergence theorem to show that 
\[
\lim_{\epsilon \rightarrow 0} \int\limits_0^1 \langle \Lambda_{(s,t)}^\epsilon , f_0\rangle ds = \int\limits_0^1 \lim_{\epsilon \rightarrow 0} \langle \Lambda_{(s,t)}^\epsilon , f_0\rangle ds = \int\limits_{0}^1 \langle \Lambda_{(s,t)} , f_0\rangle ds.
\]
Recall from proposition \ref{prop: summary of measure properties} part $(a)$ that 
\[
\left|\langle \Lambda_{(s,t)}^\epsilon, f  \right| \le C'(s,t) ||f||_\infty 
\]
and so the application of the dominated convergence theorem holds provided we show 
\[
\int\limits_0^1 C'(s,t) ds < \infty 
\]
Finiteness comes down to several computations and is proven in Proposition \ref{prop: DCT requirement} in subsection \ref{DCT verification}. 
Assuming the finiteness of this integral for now, the dominated convergence theorem gives 
\[
1\le \lim_{\epsilon \rightarrow 0} \frac{4!}{1-t} \int\limits_0^1 \langle \Lambda_{(s,t)}^\epsilon,|x_1-x_2| \rangle ds = \frac{4!}{1-t} \int\limits_0^1 \langle \Lambda_{(s,t)},|x_1-x_2| \rangle ds.
\]
However, this leads to a contradiction because, by our assumption, the integrand must be zero, which is absurd. 
\end{proof}

\subsubsection{Proof of Continuity of $F^{[f]}$}
\begin{proof}
For continuity, we show that for a fixed $f\in \mathcal{V}$ and fixed $(s,t)\in (0,1)^2\cap \Q^2$, the function
\begin{equation}\label{continuity function}
F^{[f]}(s,t) := \int\limits_{\R^4} \hat{\mu}(\xi) \prod_{i=1}^3 \hat{\mu}(\eta_i) \hat{f}(-\eta_1 - \frac{s}{1-t}\xi,-\eta_2 - \frac{1-s}{1-t}\xi, -\eta_3 + \frac{t}{1-t}\xi )d\xi d\eta_1d\eta_2d\eta_3
\end{equation}
obeys for any $\epsilon>0$, we can find a $\delta>0$ such that whenever $|(s,t)-(s',t')|<\delta$
\[
\left|  F^{[f]}(s,t)-F^{[f]}(s',t')  \right| <\epsilon.
\]
To this end, we proceed by showing that \eqref{continuity function} is finite independent of how $(s',t')\rightarrow (s,t)$ via the dominated convergence theorem.

As $(s',t')$ approaches $(s,t)$, we obtain pointwise convergence of $F^{[f]}$.
To show that \eqref{continuity function} is finite, recall the notation we had previously 
\[
\mathcal{D}_\xi = \int\limits_{\R^3}\prod_{i=1}^3 \hat{\mu}(\eta_i) \hat{f}(-\eta_1 - \frac{s}{1-t}\xi,-\eta_2 - \frac{1-s}{1-t}\xi, -\eta_3 + \frac{t}{1-t}\xi ) d\eta_1d\eta_2d\eta_3.
\]
Additionally, the bound \eqref{Dxi bound for finite fourier integral} given by 
\[
\mathcal{D}_\xi \le \norm{\hat{f}}_\infty C^3(\beta) (6R)^3 (1+c|\xi|)^{-3\beta} 
\]
is dependent upon $\beta$, $\hat{f}$ and $(s,t)$.
In particular, $c$ is dependent on $(s,t)$.
To bypass the dependence on $(s,t)$, fix a box such that $(s,t)\in [a,b]^2 \subset (0,1)^2$. 
In $[a,b]^2$, redefining $c$ by $c=\text{min}\{\frac{a}{1-a}, \frac{1-b}{1-a}\}$, one obtains
\[
D_\xi \le \norm{\hat{f}}_\infty C^3(\beta) \left(\frac{(6R)^3}{\text{min}\{\frac{a}{1-a}, \frac{1-b}{1-a}\}}\right) (1+|\xi|)^{-3\beta} .
\]
Consequently
\[
\int\limits_\R \hat{\mu}(\xi) D_\xi d\xi \le \norm{\hat{f}}_\infty C^4(\beta) \left(\frac{(6R)^3}{\text{min}\{\frac{a}{1-a}, \frac{1-b}{1-a}\}}\right) \int\limits_\R (1+|\xi|)^{-4\beta}d\xi  
\]
which is finite because $4\beta>1$. 
Thus, the dominated convergence theorem applies and continuity follows with $\delta\le |b-a|$.
\end{proof}

\subsubsection{Verification of Requirements for the dominated convergence theorem in Theorem \ref{theorem: containment result}}\label{DCT verification}

In the proof of Theorem \ref{theorem: containment result}, we use the following fact, which we now verify. 

\begin{proposition}\label{prop: DCT requirement}
The expression 
\begin{equation}\label{expression: integral of constant in terms of s}
\int\limits_0^1 C'(s,t) ds
\end{equation}
is finite.
Recall that $C'(s,t)$ is given in \eqref{the constant we integrate over} and explicitly is
\[
C'((s,t)) \approx \int\limits_\R (1+|\xi|)^{-\beta}\left(1+\frac{s}{1-t} |\xi|\right)^{-\beta}\left(1+ \frac{1-s}{1-t} |\xi|\right)^{-\beta}\left(1+ \frac{t}{1-t}|\xi|\right)^{-\beta}d\xi
\]
and that $t\in (0,1/2)$ is a fixed number while $\beta\in (1/4,1/2)$.
\end{proposition}
The computations involved in this proof and what follows make heavy use of the following fact for $\alpha>0$
\begin{equation}\label{main fact}
\frac{1}{1+\alpha |\xi|} \approx \begin{cases}
    1 & |\xi| < \frac{1}{\alpha} \\
    \frac{1}{\alpha |\xi|} & |\xi|\ge \frac{1}{\alpha}.
\end{cases}
\end{equation}
We apply this fact with $\alpha = s/(1-t), \  (1-s)/(1-t)$ and $t/(1-t)$. 

\begin{proof}
The proof follows by showing a decomposition of \eqref{expression: integral of constant in terms of s}. 
This decomposition results in three summands, and the finiteness of each of these summands is proved sequentially in the Lemmas that follow. 

Our first reduction is that $C'(s,t)$ is symmetric in $s$, so \eqref{expression: integral of constant in terms of s} can be rewritten as follows
\[
\int\limits_0^1 C'(s,t) ds = 2\int\limits_0^{1/2} C'(s,t)ds.
\]

The next reduction comes from applying \eqref{main fact} to $C'(s,t)$. 
That is, $C'(s,t)$ is an integral in $\xi$; thus for a fixed $s$ and $t$, $|\xi|$ falls into one of the following intervals
\[
R_1 \bigcup R_2 \bigcup R_3 = \left[0,\frac{1-t}{1-s}\right)  \bigcup \left[\frac{1-t}{1-s}, \frac{1-t}{s}\right)  \bigcup \left[\frac{1-t}{s},\infty\right). 
\]
Consequently, the integrand of $C'(s,t)$, which we denote as $I(\xi,s,t)$, changes based on where $|\xi|$ falls  with respect to $R_1$, $R_2$ and $R_3$
\begin{equation}\label{decomposition of C(s,t)}
I(\xi,s,t) \approx
\begin{cases}
    |\xi|^{-2\beta}\left( \frac{1-t}{t} \right)^{-\beta}  &  |\xi|\in R_1 \\
    |\xi|^{-3\beta}\left( \frac{1-t}{t} \right)^{-\beta} \left( \frac{1-s}{t} \right)^{-\beta}  & |\xi|\in R_2 \\
    |\xi|^{-4\beta}\left( \frac{1-t}{t} \right)^{-\beta} \left( \frac{1-s}{t} \right)^{-\beta} \left( \frac{s}{t} \right)^{-\beta}  & |\xi|\in R_3.
\end{cases}
\end{equation}
This allows us to rewrite \eqref{expression: integral of constant in terms of s} as
\begin{equation}\label{expression: summand of Rs}
\int\limits_0^1 C'(s,t) ds = 2\int\limits_0^{1/2}C'(s,t)ds \lesssim \sum_{j=1}^3 \int\limits_0^{1/2} \int\limits_{R_j} I(\xi,s,t) d\xi ds.
\end{equation}
Our goal becomes showing that each of the summands above is finite. 

Below, Lemma \ref{lemma: R1 summand} shows finiteness of the summand involving $R_1$, Lemma \ref{lemma: R2 summand} addresses with the summand with $R_2$ and Lemma \ref{lemma: R3 summand} address the finiteness of the summand with $R_3$.
Under these assumptions, the proof is finished. 
\end{proof}


\begin{lemma}\label{lemma: R1 summand}
From the expression \eqref{expression: summand of Rs}, the summand involving $R_1$, displayed below
\[
\int\limits_0^{1/2} \int\limits_{R_1} I(\xi,s,t)d\xi ds
\]
is finite. 
Recall that $t\in(0,1/2)$ is fixed, $R_1 = [0, (1-t)/(1-s))$ and that $\beta\in (1/4,1/2]$. 
\end{lemma}

\begin{proof}
We prove the finiteness of this stated integral in two cases: when $\beta\not=1/2$ and when $\beta=1/2$. 

When $\beta\not=1/2$, we begin by applying \eqref{main fact} and use calculus to obtain
\begin{align*}
\int\limits_0^\frac{1}{2} \int\limits_{R_1} I(\xi,s,t) d\xi ds & \approx \int\limits_0^{1/2} \int\limits_{R_1} |\xi|^{-2\beta}\left( \frac{1-t}{t} \right)^{-\beta}  d\xi ds \\
& = \left( \frac{t}{1-t} \right)^{\beta}  \left( \frac{1}{1-2\beta} \right)  \int\limits_0^\frac{1}{2} \left(\frac{t}{1-s}\right)^{1-2\beta} ds .
\end{align*}
Next, we use the coarse approximation that $\frac{t}{1-s}\le 2$ and is a continuous function for $s\in [0,1/2]$, leading to 
\[
\left( \frac{t}{1-t} \right)^{\beta}  \left( \frac{1}{1-2\beta} \right)  \int\limits_0^\frac{1}{2} \left(\frac{t}{1-s}\right)^{1-2\beta} ds 
\le \frac{t^{1-\beta}}{(1-t)^\beta (1-2\beta)} 2^{-2\beta}.
\]

When $\beta = 1/2$, the computation becomes slightly different, and we obtain
\begin{align}\label{R2 case for integrability}
\int\limits_0^\frac12 \int\limits_{R_1} I(\xi) d\xi ds &\le \int\limits_0^\frac12 \int\limits_0^{\frac{t}{1-s}} \left( \frac{1}{1+|\xi|}   \right)^{\frac{1}{2}} \left(  \frac{ 1}{ 1+ \frac{1-t}{t}|\xi |}\right)^{\frac{1}{2}} d\xi ds \notag \\
&\le \left( \frac{t}{1-t}  \right)^{\frac{1}{2}} \int\limits_0^\frac12 \int\limits_0^{\frac{t}{1-s}} \frac{d\xi }{\frac{t}{1-t} + |\xi|} ds \notag  \\
&= \left( \frac{t}{1-t}  \right)^{\frac{1}{2}}\left[ \int\limits_0^\frac12 \log\left(\frac{t}{1-t} +\frac{t}{1-s} \right) ds+ \frac{1}{2}\log \frac{1-t}{t}\right].
\end{align}
Now, we change variables $s\mapsto 1-  \frac{t}{x-\frac{t}{1-t}}$ and obtain
\begin{equation}\label{change of variables for integrability}
\eqref{R2 case for integrability}= \left( \frac{t}{1-t}  \right)^{\frac{1}{2}}
\int\limits_{\frac{t}{1-t} +t}^{\frac{t}{1-t} + 2t} \left(\frac{t}{x-\frac{t}{1-t}}\right)^2 \log x dx 
 + \frac{1}{2}  \left( \frac{t}{1-t}  \right)^{\frac{1}{2}} \log \frac{1-t}{t} .
\end{equation}
To conclude, observe for $x\in \left[ t/(1-t) + t , t/(1-t) +2t  \right]$, $x-t/(1-t) \ge t$ hence
\begin{equation}\label{crude bounds for integrability}
\left(\frac{t}{x-\frac{t}{1-t}} \right)^2\le 1 \quad 
\text{ and } \quad 
\log x \le \log \left(\frac{t}{1-t} +2t\right).
\end{equation}
Plugging \eqref{crude bounds for integrability} into \eqref{change of variables for integrability} provides
\[
\eqref{R2 case for integrability} \le \left( \frac{t}{1-t}  \right)^{\frac{1}{2}} t
\log \left(\frac{t}{1-t} +2t\right)
+ \frac{1}{2}  \left( \frac{t}{1-t}  \right)^{\frac{1}{2}} \log \frac{1-t}{t} .
\]
\end{proof}


\begin{lemma}\label{lemma: R2 summand}
From the expression \eqref{expression: summand of Rs}, the summand involving $R_2$, displayed below
\[
\int\limits_0^{1/2} \int\limits_{R_2} I(\xi,s,t)d\xi ds
\]
is finite. 
Recall that $t\in(0,1/2)$ is fixed, $R_2 = [(1-t)/(1-s), (1-t)/s )$ and that $\beta\in (1/4,1/2]$. 
\end{lemma}

\begin{proof}
After applying \eqref{main fact} to the above integral, three cases must be considered: $\beta \in [1/4,1/3)$, $\beta =1/3$ and $\beta \in (1/3,1/2]$. 
In each of these subcases, the inner-most integral behaves differently, so we show these estimates first and then evaluate the integral with respect to $s$. 

Considering the integral in $\xi$, we have the following approximations from calculus
\begin{equation}\label{expression: dealing with the beta =1/3 singularity}
\int\limits_{(1-t)/(1-s)}^{(1-t)/s} I(\xi ,s,t,) d\xi \le \frac{t^{2\beta}}{(1-t)^\beta} \times 
\begin{cases}
\frac{1}{1-3\beta} \left[ (\frac{1-t}{s})^{1-3\beta} - (\frac{1-t}{1-s})^{1-3\beta } \right]  & 1/4 < \beta < 1/3 \\
\log \left[\frac{1-s}{s}  \right] & \beta = 1/3 \\
\frac{1}{3\beta - 1 } \left[ (\frac{1-s}{1-t})^{3\beta-1} - (\frac{s}{1-t})^{3\beta -1} \right]  & 1/3 < \beta < 1/2 
\end{cases}.
\end{equation}

Now, we sequentially evaluate the integral over $s$ for the differing cases.
When $1/4< \beta<1/3$ and using the estimate for this case given in \eqref{expression: dealing with the beta =1/3 singularity}, the integral in question becomes
\[
\int\limits_0^{1/2} \int\limits_{R_2} I(\xi ,s,t)d\xi ds \le \frac{t^{2\beta}}{(1-3\beta)(1-t)^\beta} \int\limits_0^{1/2} (1-s)^{-\beta} \left[ \left(\frac{1-t}{s}\right)^{1-3\beta} - \left(\frac{1-t}{1-s}\right)^{1-3\beta } \right] ds.
\]
For $s\in (0,1/2)$, we estimate $(1-s)^{-\beta} \le 2^\beta$ so that we can estimate as follows
\[
\int\limits_0^{1/2} \int\limits_{R_2} I(\xi ,s,t)d\xi ds \le \frac{2^{\beta}t^{2\beta}}{(1-3\beta)(1-t)^\beta} \int\limits_0^{1/2} \left[ \left(\frac{1-t}{s}\right)^{1-3\beta} - \left(\frac{1-t}{1-s}\right)^{1-3\beta } \right] ds.
\]
Using calculus, one can evaluate the above integral to obtain
\[
\int\limits_0^{1/2} \int\limits_{R_2} I(\xi ,s,t)d\xi ds  \le 
\frac{2^{\beta}t^{2\beta}(1-t )^{1-3\beta}}{(1-3\beta)(1-t)^\beta 3\beta}< \infty. 
\]

When dealing with the case $\beta =1/3$, applying \eqref{expression: dealing with the beta =1/3 singularity} results in
\[
\int\limits_0^{1/2} \int\limits_{R_2} I(\xi ,s,t)d\xi ds \le \frac{t^{2\beta}}{(1-t)^\beta} \int\limits_0^{1/2} \log \left(\frac{1-s}{s}\right)ds .
\]
We apply two estimates to evaluate this integral: $s\in (0,1/2)$ yields $(1-s)^{-\beta } \le 2^\beta$, and the change of variables $s\mapsto (x+1)^{-1}$ results in 
\[
\int\limits_0^{1/2} \int\limits_{R_2}I(\xi ,s,t) d\xi \le \frac{t^{2\beta}}{(1-t)^\beta}  \int\limits_1^\infty 2^\beta \frac{\log x }{(1+x^2)}dx  \le  \frac{t^{2\beta } 2^\beta \log 2}{(1-t)^\beta } < \infty .
\]

The final case over $1/3<\beta<1/2$ is very similar to the case over $1/4<\beta<1/3$, so we use similar techniques. Again, \eqref{expression: dealing with the beta =1/3 singularity} gives us the expression 
\[
\int\limits_0^{1/2} \int\limits_{R_2} I(\xi,s,t) d\xi ds \le \frac{t^{2\beta}}{(3\beta-1)(1-t)^\beta} \int\limits_0^{1/2} (1-s)^{-\beta} \left[ \left(\frac{1-s}{1-t}\right)^{3\beta-1} - \left(\frac{s}{1-t} \right)^{3\beta-1 } \right] ds.
\]
This computation again uses for $s\in(0,1/2)$, $(1-s)^{-\beta }\le 2^\beta $ and then proceeds with calculus to obtain
\[
\int\limits_0^{1/2} \int\limits_{R_2} I(\xi,s,t) d\xi ds \le \frac{t^{2\beta}2^\beta }{3\beta (3\beta-1)(1-t)^\beta }\left[ 1- 2^{1- 3\beta} \right]<\infty  . 
\]

\end{proof}


\begin{lemma}\label{lemma: R3 summand}
From the expression \eqref{expression: summand of Rs}, the summand involving $R_3$, displayed below
\[
\int\limits_0^{1/2} \int\limits_{R_3} I(\xi,s,t)d\xi ds
\]
is finite. 
Recall that $t\in(0,1/2)$ is fixed, $R_3 = [(1-t)/s , \infty  )$ and that $\beta\in (1/4,1/2]$. 
\end{lemma}

\begin{proof}
Applying \eqref{main fact} to the integral in question one yields
\[
\int\limits_0^{1/2} \int\limits_{R_3} I(\xi,s,t) d\xi ds \le \frac{t^{3\beta}}{(1-t)^\beta } \int\limits_0^{1/2} \frac{1}{s^\beta (1-s)^\beta}\int\limits_{R_3} |\xi|^{-4\beta } d\xi ds. 
\]
Since $\beta >1/4$, we can use calculus to evaluate the integral in $\xi$,
\[
\frac{t^{3\beta}}{(1-t)^\beta } \int\limits_0^{1/2} \frac{1}{s^\beta (1-s)^\beta}\int\limits_{R_3} |\xi|^{-4\beta } d\xi ds \le 
\frac{ (1-t)^{1-5\beta} t^{3\beta}}{1+4\beta  } \int\limits_0^{1/2} \frac{s^{3\beta -1 }}{(1-s)^\beta}  ds . 
\]
By applying the bound $(1-s)^{-\beta} \le 2^\beta$ for $s\in(0,1/2)$, we can compute the integral in $s$ using calculus to yield
\[
\int\limits_0^{1/2} \int\limits_{R_3} I(\xi,s,t) d\xi ds \le \frac{t^{3\beta }}{3\beta 2^{3\beta} (1+4\beta ) (1-t)^{5\beta-1  }}< \infty. 
\]
\end{proof}

\section{Avoiding Set} 

In this section, we focus on the construction of a set that avoids nontrivial solutions to a given equation as in \eqref{4variable equation}.
To do this, we follow a construction in \cite{pablo} which, for completeness, we recall the important components in the following subsection. 
Briefly, the first step partitions $[0,1]$ into $N_1$ equal length intervals. 
Then, let $Q_{N_1}$ be an avoiding set in the integers, whose existence is established in \cite{ruzsa1}, and pick the corresponding subintervals of $[0,1]$.
That is, the union of intervals of the form 
\[
\left[\frac{q}{N_1}, \frac{q+1}{N_1}\right] \text{ for all }q\in Q_{N_1}. 
\]
Then, according to some random shift $l_1$, we keep the shifted intervals associated with $u=q+l_1\mod N_1$ for all $q\in Q_{N_1}$, i.e., the union of all such intervals of the form $[\frac{u}{N_1}, \frac{u+1}{N_1}]$. 
At the $n$th stage of construction, we follow a similar algorithm. 
Each of the basic Cantor intervals given by the $n-1$ stage of construction are partitioned into $N_{n}$ many equal-length pieces. 
For $N_{n}$ take an avoiding set $Q_{N_n}$ that initially dictates the children of each basic Cantor interval.
However, for \textit{each} of the basic Cantor intervals at the $n-1$ stage, there is an associated shift for its children that survived construction. 
That is, the children of each basic Cantor interval do not necessarily have the same orientation. 
These random shifts ensures that our set has our desired Fourier decay. 

\subsection{Notation \& Construction}

Let $\{L_n \},\{M_n\}\subset \N$ be fixed sequences of positive integers with $1\le L_n\le M_n$ and $M_n\ge 2$ for all $n\in \N$. Set 
\[
\Sigma_n := \{ \mathbf{j}:= (j_1,\dots, j_n): j_k\in M_k \text{ for each } k\in \{1,\dots,n\}   \} \text{ and } \Sigma^* = \bigcup_{n=0}^\infty \Sigma_n.
\]
For each $M_n$ and multi-index $\mathbf{j}=(j_1,\dots ,j_n)\in \Sigma_n$ there is an associated interval in $[0,1]$ as follows,
\[
I_{M_n, \mathbf{j}} := \sum_{k=1}^n \frac{j_k}{M_1 \dots M_k} + \left[0,\frac{1}{M_1\cdots M_n}\right].
\]
The Cantor iterates, $E_n$, to be specified consists of a finite union of the intervals $I_\mathbf{j}$ for $\mathbf{j}\in \Sigma_n$; thus,
\[
E_1 \supsetneq E_2 \supsetneq \cdots \supsetneq E_n \supsetneq E_{n+1} \supsetneq \cdots \supsetneq E, \quad E = \bigcap_{n=1}^\infty E_n. 
\]

To specify $E_n$, we introduce our random pieces.  
Let $\{X_\mathbf{j}: \mathbf{j}\in \Sigma^*\}$ be a family of independent random sets that obey the following properties:
\begin{enumerate}\label{properties for pablos construction}
    \item For each $n\ge 0$ and $\mathbf{j}\in \Sigma_n$, the set $X_\mathbf{j}$ is a subset of $[M_{n+1}]$ with $\#(X_j)=L_{n+1}$, almost surely. 

    \item For each $a\in [M_{n+1}]$ and $\mathbf{j}\in \Sigma_n$, $\mathbb{P}(a\in X_\mathbf{j})=\frac{L_{n+1}}{M_{n+1}}$. 
\end{enumerate}
Each sequence of random sets $\{X_\mathbf{j}: \mathbf{j}\in \Sigma^*\}$ leads to a sequence of multi-indices $\{\mathcal{J}_n\}$. 
That is, for every $n\ge 1$ define,
\begin{equation}
\mathcal{J}_n := \{(j_1,\dots ,j_n)\in \Sigma_n : j_{k+1}\in X_{(j_1,\dots, j_k)} \text{ for all } k=0,\dots,n-1   \} \subset \Sigma_n . 
\end{equation}
With this collection of multi-indices, we may define the Cantor iterates $E_n$ as 
\[
E_n = \bigcup_{j\in \mathcal{J}_n} \text{ and hence } E=\bigcap_{n=1}^\infty E_n .
\]

Because we are interested not only in a set but also in the types of measures that are supported on it, we include the description of a measure such as that in \cite{pablo} which has ideal Fourier decay. 
From properties 1 and 2, we have  $|\mathcal{J}_n| = |L_1 \cdots L_n|$ almost surely and that for every fixed multi-index $\mathbf{j}\in \Sigma_n$,
\begin{equation}\label{definition of the probability for second property for pablos construction}
\mathbb{P} ((\mathbf{j},j_{n+1})\in \mathcal{J}_{n+1}) = \mathbf{1}_{\mathcal{J}_n}(\mathbf{j}) \frac{L_{n+1}}{M_{n+1}} 
\end{equation}
for all $j_{n+1} \in \{0,1,\dots , M_{n+1}-1\}$. 
On $E_n$, define the following probability density function 
\[
\mu_n(x) := \frac{M_1\cdots M_n}{L_1 \cdots L_n} \sum_{\mathbf{j}\in \mathcal{J}_n} \mathbf{1}_{I_{\mathbf{j}}}(x).
\]
By the Carath\'{e}odory extension theorem, there exists a probability measure $\mu$ supported on $E$ such that 
\[
\mu(I_\mathbf{j}) = \mu_n (I_\mathbf{j}) = \frac{1}{L_1 \cdots L_n} \text{ for all } \mathbf{j}\in \mathcal{J}_n.
\]
Specifically, $\mu_n\rightarrow \mu $ weakly. 
Even luckier for us, the authors of \cite{pablo} prove that such a construction above will lead to a measure $\mu$ that has maximal Fourier decay. 
This is the theorem below; for proof, we refer the reader to \cite{pablo}. 

\begin{theorem}[\cite{pablo}, Theorem 2.1]\label{pablo theorem for fourier decay}
Suppose that $\{M_n\}\subset \N$ satisfies 
\begin{equation}\label{first condition in pablos theorem}
\lim_{n\rightarrow \infty } \frac{\log M_{n+1}}{\log (M_1\cdots M_N)} =0.
\end{equation}
Fix $\sigma >0$ such that 
\begin{equation}\label{second sigma condition for pablo theorem}
\sigma < \liminf_{n\rightarrow \infty } \frac{\log (L_1\cdots L_n)}{\log (M_1\cdots M_n}.
\end{equation}
Then for $\mu$ as above there almost surely exists a constant $C>0$ such that 
\[
|\hat{\mu(k)}| \le C|k|^{-\sigma/2} \text{ for all } k\in \Z\setminus\{0\}. 
\]
\end{theorem}

\subsection{Transference Lemma}\label{subsection: transference}
This section is dedicated to proving the following Lemma, which is then utilized to prove Theorem \ref{E For avoidance}.

\begin{lemma}\label{partial avoidance in the real line}
Let $N$ be a large positive integer. There exists a constant $\alpha$ independent of $N$ and a set $U\subset [N]$ such that 
\begin{itemize}
    \item $\#(U) \gtrsim \sqrt{N}e^{-\alpha \sqrt{\log N}}$
    \item For any $l\in \N$, if there exists $\{x,y,u,v\}\subset \mathbb{I}(U+l)$ such that $f(x,y,u,v)=0$, then there exists a $j'\equiv j+l\mod N$ such that $\{x,y,u,v\}\subset \frac{j'}{N}+[0,\frac{1}{N}]$.
\end{itemize}
\end{lemma}

\begin{proof}
We show each condition separately. 
The first follows by a construction in \cite{ruzsa1} while the second follows from some bounds we establish. 
Let $N$ be such that $N\gg a+b$ and $S=a+b=c+d=$ for $\{a,b,c,d\}$ as the coefficients of the equation in \eqref{4variable equation}. 
We also rewrite such an expression as 
\[
ax+by=cu+dv \iff f(x,y,u,v) = ax+by-cu-dv
\]
since the right side would be more convenient notation occasionally in this proof. 

To show the first condition, set $M=\lceil N/(S^2 + abcd)\rceil>4$. 
A construction in \cite{ruzsa1} guarantees the existence of a set $V\subset [M]$ such that if $\{x,y,u,v\}\subset V$ and satisfies \eqref{4variable equation}, then it must be the case that $x=y=u=v$.
Set $U:= (S+1)V= \{(S+1)v : v\in V \}$ so that the bound shown in \cite{ruzsa1} on the size of $U$ obeys $\#(U) = \#(V) \gtrsim \sqrt{N}e^{-\alpha \sqrt{\log N}}$, as desired.

For the second condition, let $\{x,y,u,v\}\subset \mathbb{I}(U+l)$ for some $l\in \N$. 
By the construction of the set $\mathbb{I}(U+l)$, each $x,y,u,v$ must have a real number expression as follows
\begin{equation}\label{real number form }
x = \frac{(S+1)j_x +l +\delta_x}{N} - \chi_x \text{ where } \chi_x = \begin{cases}
    0 &  (S+1)j_x+l<N \\
    1 & (S+1)j_x+1\ge N 
\end{cases}
\end{equation}
where $\delta_x\in [0,1)$, $(S+1)j_x+l\in U+l\mod N$ and $j_x\in V$; similarly for $y,u,v$. 
Plugging the form \eqref{real number form }, into equation $f$ we obtain
\begin{equation}\label{straight plugging in}
(S+1)(aj_x +bj_y - cj_u - dj_v) + a \delta_x +b\delta _y -c\delta_u -d\delta_v - N(a\chi_x + b\chi_y -c\chi_u - d \chi_v)=0. 
\end{equation}
We claim that \eqref{straight plugging in} leads to $aj_x+bj_y=cj_u+dj_v $, by which the construction in \cite{ruzsa1} requires $j_x=j_u=j_u=j_v$. 
Hence, we will necessarily have that $\{x,y,u,v\}\subset \frac{j'}{N} + [0,\frac{1}{N}]$ and that would conclude the proof. 

To show the claim, we proceed by contradiction and assume that $aj_x+bj_y\not=cj_u+dj_v $.
Observe this first bound
\begin{equation}\label{trivial lower bound}
(S+1)|aj_x+bj_y - cj_u - dj_v| \ge (S+1)|(aj_x +b j_y)-(cj_u+dj_v)| \ge S+1
\end{equation}
and this second bound
\begin{equation}\label{triangle inequality for deltas}
|a\delta_x + b\delta_y -c\delta_u-d \delta_v| \le \max\{a\delta_x +b\delta_y, c\delta_u + d\delta_v\} \le S.
\end{equation}
Together, these bounds dictate that
\[
a\chi_x + b\chi_y -c\chi_u -d\chi_v\not=0.
\]
Therefore, we must have $N|a\chi_x +b\chi_y -c\chi_u-d\chi_v|\ge N$.
Furthermore, the triangle inequality and \eqref{triangle inequality for deltas} provide
\[
|S(aj_x +bj_y - cj_u - dj_v) + a \delta_x +b\delta _y -c\delta_u -d\delta_v | \le S\max\{aj_x +bj_y, cj_u +dj_v\} + S \le  \frac{N}{2S} + S.
\]
However, the above and $N|a\chi_x +b\chi_y -c\chi_u-d\chi_v|\ge N$ contradict \eqref{straight plugging in} for large enough $N$; thus, \eqref{straight plugging in} implies that $aj_x+bj_y=cj_u+dj_v$.
\end{proof}

\subsection{Proof of Theorem \ref{E For avoidance}}
\begin{proof}
We proceed in two steps. The first step is to specify the necessary sequences and random variables so that $E$ satisfies theorem \ref{pablo theorem for fourier decay} and has positive Fourier dimension. 
Second, that these choices ensure that $E$ avoids nontrivial solutions to $f$. 

For every $n\in \N$ set
\[
M_n = 4 + n, \ U_{M_n}\subset[M_n] \text{ and }   L_n = \#(U_{M_n})
\] 
where $U_{M_n}$ is a set that exists via Lemma \ref{partial avoidance in the real line}. 

That these choices satisfy Theorem \ref{pablo theorem for fourier decay}, we check that both \eqref{first condition in pablos theorem} and \eqref{second sigma condition for pablo theorem}. 
The first follows by 
\[
\lim_{n\rightarrow \infty } \frac{\log M_{n+1}}{\log(M_1\cdots M_n)} = \lim_{n\rightarrow \infty } \frac{\log(n+5)}{\log (n!) - \log 5!} = 0.
\]
To observe \eqref{second sigma condition for pablo theorem} is possible for $\sigma = \frac{1}{2}-\epsilon$ for all $\epsilon>0$ observe the limit as follows
\[
\lim_{n\rightarrow \infty } \frac{\log (L_1 \cdots L_n)}{\log(M_1 \cdots M_n)} \lesssim \lim_{n\rightarrow \infty } \frac{\log \sqrt{n!}}{\log n!} =\frac12.
\]
Therefore, as long as we can supply a way to specify the random variables $X_\mathbf{j}$, we will construct a set $E$ that supports a measure with Fourier decay. 

For any $\mathbf{j}\in \Sigma^*$, let $\{l_\mathbf{j}\}$ be a family of independent random variables distributed uniformly among $[M_{n+1}]$ when any $\mathbf{j}\in \Sigma_n$. 
Using these $l_\mathbf{j}$, we set
\[
X_\mathbf{j} = U_{M_{n+1}} +l_\mathbf{j} \mod M_{n+1} 
\]
for any $\mathbf{j}$ so that $\{X_\mathbf{j}\}$ is a collection of independent random variables.
Finally, we need to check that this collection obeys properties stated above \eqref{properties for pablos construction}. 
The first follows because $\#(X_\mathbf{j})=\#(U_{M_{n+1}} )= L_{n+1}$. 
The second follows by \eqref{definition of the probability for second property for pablos construction}.
Therefore, we conclude that a set $E$ obtained via the construction described in section 3.1 with the measure $\mu$ obeys the conditions of said theorem and hence has the stated Fourier dimension. 

To conclude, we only need to show that this construction avoids nontrivial solutions of $f$.
To this end, let $\{x,y,u,v\}\subset E$ such that $f(x,y,u,v)=0$. 
We also assume that not all $\{x,y,u,v\}$ are identical; hence, there exists a maximum index $n$ such that $\{x,y,u,v\}\subset I_\mathbf{j} \subset E_n$ for some $\mathbf{j} \in \Sigma_n$.
According to construction, $I_\mathbf{j}$ is partitioned into $M_{n+1}$ many pieces, and only the $U_{M_{n+1}}+l_\mathbf{j} \mod M_{n+1}$ pieces survive construction. 
Observe that 
\[
\{x,y,u,v\} \subset \alpha_\mathbf{j} + \bigcup_{j_{n+1} \in U_{M_{n+1}} + l_\mathbf{j} \mod M_{n+1}} \left[\frac{j_{n+1}}{M_1 \cdots M_{n+1}} , \frac{j_{n+1}+1}{M_1 \cdots M_{n+1}}  \right] 
\]
where $\alpha_\mathbf{j}$ is the leftmost endpoint of $I_\mathbf{j}$. 
Consider the following affine transformation $A$ that depends on $\mathbf{j}$ 
\begin{align*}
A:\mathbb{I}(U_{M_{n+1}} +l_\mathbf{j}) &\rightarrow \alpha_\mathbf{j} + \bigcup_{j_{n+1} \in U_{M_{n+1}} + l_\mathbf{j} \mod M_{n+1}} \left[\frac{j_{n+1}}{M_1 \cdots M_{n+1}} , 
\frac{j_{n+1}+1}{M_1 \cdots M_{n+1}}  \right],\\ 
 A:x &\mapsto \frac{x}{M_1\cdots M_n} + \alpha_\mathbf{j}.
\end{align*}
In particular, this means that there is an affine copy of $\{x,y,u,v\}$ inside $\mathbb{I}(U_{M_{n+1}} +l_\mathbf{j})$. 
By Lemma \ref{partial avoidance in the real line}, there must be an index $j_0 \in U_{M_{n+1}}$ so that this affine copy of $\{x,y,u,v\}$ lies in a single interval. 
Thus, $\{x,y,u,v\}\subset I_{\mathbf{j_{n+1}}}$ where $\mathbf{j}_{n+1} = (\mathbf{j}_n , j_0)$.
However, this contradicts the maximality of the index $n$ so that any such $\{x,y,u,v\}\not\subset E$. 
\end{proof}

\end{document}